\documentclass{amsart}
\usepackage{amstext,amssymb,amsthm,amsopn,newlfont,graphpap,graphics,graphicx,mathrsfs,enumitem}
\allowdisplaybreaks
\usepackage[parfill]{parskip}
\usepackage[noadjust]{cite}
\usepackage{epigraph}
\usepackage[colorlinks=true,
            linkcolor=red,
            urlcolor=blue,
            citecolor=magenta]{hyperref}
\usepackage{color}
\usepackage{mathrsfs}
\allowdisplaybreaks
\theoremstyle{plain}
\newtheorem{thm}{Theorem}[section]
\newtheorem*{thm*}{Theorem}
\newtheorem{prop}{Proposition}[section]
\newtheorem*{prop*}{Proposition}
\newtheorem{cor}{Corollary}[section]
\newtheorem*{cor*}{Corollary}

\newtheorem*{lem*}{Lemma}
\theoremstyle{definition}
\newtheorem{defn}{Definition}[section]
\newtheorem*{defn*}{Definition}

\newtheorem*{exmp*}{Example}

\newtheorem*{exmps*}{Examples}

\newtheorem{rem}{Remark}[section]
\newtheorem*{rem*}{Remark}

\newtheorem*{rems*}{Remarks}

\newtheorem*{note*}{Note}
\newcommand{\Z}{{\mathbb Z}}
\newcommand{\R}{{\mathbb R}}
\newcommand{\C}{{\mathbb C}}

\DeclareMathOperator{\Rep}{Re}

\DeclareMathOperator{\orb}{orb}

\begin{document}
\title[On non-hypercyclicity of scalar type spectral operators]
{On the non-hypercyclicity\\
of scalar type spectral operators\\
and collections of their exponentials}
\author[Marat V. Markin]{Marat V. Markin}
\address{
Department of Mathematics\newline
California State University, Fresno\newline
5245 N. Backer Avenue, M/S PB 108\newline
Fresno, CA 93740-8001
}
\email{mmarkin@csufresno.edu}
\begin{abstract}
Generalizing the case of a normal operator in a complex Hilbert space, we give a straightforward proof of the non-hypercyclicity of a \textit{scalar type spectral operator} $A$ in a complex Banach space as well as of the collection $\left\{e^{tA}\right\}_{t\ge 0}$ of its exponentials, which, under a certain condition on the spectrum of the operator $A$, coincides with the $C_0$-semigroup generated by $A$. The spectrum
of $A$ lying on the imaginary axis, we also show that non-hypercyclic is the strongly continuous group 
$\left\{e^{tA}\right\}_{t\in {\mathbb R}}$ of bounded linear operators generated by $A$. From the general results, we infer that, in the complex Hilbert space $L_2({\mathbb R})$, the anti-self-adjoint differentiation operator $A:=\dfrac{d}{dx}$ with the domain $D(A):=W_2^1({\mathbb R})$ is non-hypercyclic and so is the left-translation strongly continuous unitary operator group generated by $A$.
\end{abstract}
\maketitle

\section[Introduction]{Introduction}

The concept of \textit{hypercyclicity}, underlying the theory of linear chaos, traditionally considered for \textit{continuous} linear operators on Fr\'echet spaces, in particular for \textit{bounded} linear operators on Banach spaces, and known to be a purely infinite-dimensional phenomenon (see, e.g., \cite{Grosse-Erdmann-Manguillot,Guirao-Montesinos-Zizler,Rolewicz1969}), is extended in \cite{B-Ch-S2001,deL-E-G-E2003} to \textit{unbounded} linear operators in Banach spaces, where also found are sufficient conditions for unbounded hypercyclicity and certain examples of hypercyclic unbounded linear differential operators.

\begin{defn}[Hypercyclicity]\ \\
Let
\[
A:X\supseteq D(A)\to X
\]
be a (bounded or unbounded) linear operator in a (real or complex) Banach space $X$ with a domain $D(A)$. 

A nonzero vector 
\begin{equation*}
f\in C^\infty(A):=\bigcap_{n=0}^{\infty}D(A^n)
\end{equation*}
($A^0:=I$, $I$ is the \textit{identity operator} on $X$)
is called \textit{hypercyclic} if its \textit{orbit}
under $A$
\[
\orb(f,A):=\left\{A^nf\right\}_{n\in\Z_+}
\] 
($\Z_+:=\left\{0,1,2,\dots\right\}$ is the set of nonnegative integers) is dense in $X$.

Linear operators possessing hypercyclic vectors are
said to be \textit{hypercyclic}.

More generally, a collection $\left\{T(t)\right\}_{t\in J}$ ($J$ is a nonempty indexing set) of linear operators in $X$ is called \textit{hypercyclic} if it possesses \textit{hypercyclic vectors}, i.e., such nonzero vectors $\displaystyle f\in \bigcap_{t\in J}D(T(t))$, whose \textit{orbit} 
\[
\left\{T(t)f\right\}_{t\in J}
\]
is dense in $X$.
\end{defn}

Cf. \cite{Markin2018(9),Markin2018(10)}.

As is easily seen, in the definition of hypercyclicity for a linear operator, the underlying space must necessarily be \textit{separable}. 

It is noteworthy that, for a hypercyclic linear operator $A$, the set $HC(A)$ of all its hypercyclic vectors, containing the dense orbit of any vector hypercyclic under $A$, is dense in $(X,\|\cdot\|)$, and hence, the more so, is the subspace $C^\infty(A)\supseteq HC(A)$.

Bounded \textit{normal operators} on a complex Hilbert space are known to be non-hypercyclic \cite[Corollary $5.31$]{Grosse-Erdmann-Manguillot}. In \cite{Mark-Sich2019(1)}, non-hypercyclicity is shown to hold for arbitrary normal operators (bounded or unbounded), certain collections of their exponentials, and symmetric operators.

Here, abandoning the comforts of a Hilbert space setting with its inherent orthogonality and self-duality, while generalizing non-hypercyclicity from normal to scalar type spectral operators, we furnish a straightforward proof of the non-hypercyclicity of an arbitrary \textit{scalar type spectral operator} $A$ (bounded or unbounded) in a complex Banach space as well as of the collection $\left\{e^{tA}\right\}_{t\ge 0}$ of its exponentials (see, e.g., \cite{Dunford1954,Survey58,Dun-SchIII}), which, provided the spectrum $\sigma(A)$ of the operator $A$ is located in a left half-plane
\[
\left\{\lambda\in \C\,\middle|\,\Rep\lambda\le \omega \right\}
\]
with some $\omega\in\R$, coincides with the $C_0$-\textit{semigroup} generated by $A$ 
\cite{Markin2002(2)} (see also \cite{Berkson1966,Panchapagesan1969}). The spectrum
of $A$ lying on the imaginary axis $i\R$ ($i$ is the \textit{imaginary unit}), we also show that non-hypercyclic is the strongly continuous group 
$\left\{e^{tA}\right\}_{t\in {\mathbb R}}$ of bounded linear operators generated by $A$. From the general results, we immediately infer that, in the complex Hilbert space $L_2({\mathbb R})$, the \textit{anti-self-adjoint} differentiation operator $A:=\dfrac{d}{dx}$ with the domain
\[
W_2^1({\mathbb R}):=\left\{f\in L_2(\R)\middle|f(\cdot)\ 
\text{is \textit{absolutely continuous} on $\R$ with}\ f'\in L_2(\R) \right\}
\] 
is non-hypercyclic and so is the left-translation strongly continuous unitary operator group generated by it \cite{Engel-Nagel,Hille-Phillips,Stone1932}.

\section[Preliminaries]{Preliminaries}

More extensive preliminaries concerning the \textit{scalar-type spectral operators} in complex Banach spaces, which, in particular, encompass \textit{normal operators} in complex Hilbert spaces \cite{Wermer} (see also \cite{Dun-SchII,Plesner}), can be found in the corresponding section of \cite{Markin2018(3)} (see also \cite{Dunford1954,Survey58,Dun-SchIII}). Here, we outline only a few facts indispensable for our subsequent discourse. 

With a {\it scalar type spectral operator} $A$ in a complex Banach space $(X,\|\cdot\|)$ associated are its \textit{spectral measure} (the \textit{resolution of the identity}) $E_A(\cdot)$, whose support is the spectrum $\sigma(A)$ of $A$, and the so-called {\it Borel operational calculus} assigning to any Borel measurable function $F:\sigma(A)\to \C$ a scalar type spectral operator
\begin{equation*}
F(A):=\int\limits_{\sigma(A)} F(\lambda)\,dE_A(\lambda)
\end{equation*}
(see \cite{Survey58,Dun-SchIII}).

In particular,
\begin{equation*}
A^n=\int\limits_{\sigma(A)} \lambda^n\,dE_A(\lambda),\ n\in\Z_+,
\end{equation*}
and
\begin{equation*}
e^{tA}:=\int\limits_{\sigma(A)} e^{t\lambda}\,dE_A(\lambda),\ t\in \R.
\end{equation*}

Provided
\[
\sigma(A)\subseteq \left\{\lambda\in\C\,\middle|\, \Rep\lambda\le \omega\right\},
\] 
with some $\omega\in \R$, the collection 
of exponentials $\left\{e^{tA}\right\}_{t\ge 0}$
coincides with the $C_0$-\textit{se\-migroup} generated by $A$ {\cite[Proposition $3.1$]{Markin2002(2)}} (see also \cite{Berkson1966,Panchapagesan1969}), and hence, if
\[
\sigma(A)\subseteq \left\{\lambda\in\C\,\middle|\, -\omega\le \Rep\lambda\le \omega\right\},
\] 
with some $\omega\ge 0$, the collection of exponentials
$\left\{e^{tA}\right\}_{t\in\R}$ 
coincides with the \textit{strongly continuous group} of bounded linear operators generated by $A$.

The orbit maps
\begin{equation}\label{expf1}
y(t)=e^{tA}f,\ t\ge 0,f \in \bigcap_{t\ge 0}D(e^{tA}),
\end{equation}
describe all \textit{weak/mild solutions} of the abstract evolution equation
\begin{equation}\label{+}
y'(t)=Ay(t),\ t\ge 0,
\end{equation}
\cite[Theorem $4.2$]{Markin2002(1)}, whereas the orbit maps
\begin{equation*}
y(t)=e^{tA}f,\ t\in \R,f \in \bigcap_{t\in \R}D(e^{tA}),
\end{equation*}
describe all \textit{weak/mild solutions} of the abstract evolution equation
\begin{equation}\label{1}
y'(t)=Ay(t),\ t\in \R,
\end{equation}
\cite[Theorem $7$]{Markin2018(3)} (see also \cite{Ball}). Such generalized solutions need not be differentiable in the strong sense and encompass the \textit{classical} ones, strongly differentiable and satisfying the corresponding equations in the traditional plug-in sense (cf. {\cite[Ch. II, Definition 6.3]{Engel-Nagel}}, see also {\cite[Preliminaries]{Markin2018(2)}}).

The subspaces
\[
C^\infty(A),\ \displaystyle \bigcap_{t\ge 0}D(e^{tA}),\ \text{and}\ \bigcap_{t\in \R}D(e^{tA}) 
\]
of all possible initial values for the orbits 
under $A$, $\left\{e^{tA}\right\}_{t\ge 0}$, and $\left\{e^{tA}\right\}_{t\in \R}$ are \textit{dense} in $(X,\|\cdot\|)$ since they contain the subspace
\begin{equation*}
\bigcup_{\alpha>0}E_A(\Delta_\alpha)X,\ \text{where}\ \Delta_\alpha:=\left\{\lambda\in\C\,\middle|\,|\lambda|\le \alpha \right\},\ \alpha>0,
\end{equation*}
which is dense in $(X,\|\cdot\|)$ and coincides with the class ${\mathscr E}^{\{0\}}(A)$ of the \textit{exponential type entire} vectors of the operator $A$ \cite{Markin2015} (see also  \cite{Radyno1983(1)}).

Due to its strong countable additivity, the spectral measure $E_A(\cdot)$ is bounded, i.e., there exists such an $M\ge 1$ that, for any Borel set $\delta\subseteq \C$,
\begin{equation}\label{bounded}
\|E_A(\delta)\|\le M
\end{equation}
\cite{Dun-SchI,Dun-SchIII}, the notation $\|\cdot\|$ being used here to designate the norm in the space $L(X)$ of all bounded linear operators on $X$. Adhering to this rather conventional economy of symbols hereafter, we also adopt the same notation for the norm in the dual space $X^*$.

For arbitrary $f\in X$ and $g^*\in X^*$, the \textit{total variation measure} $v(f,g^*,\cdot)$ of the complex-valued Borel measure $\langle E_A(\cdot)f,g^* \rangle$ ($\langle\cdot,\cdot\rangle$ is the {\it pairing} between the space $X$ and its dual $X^*$) is a {\it finite} positive Borel measure with
\begin{equation}\label{tv}
v(f,g^*,\C)=v(f,g^*,\sigma(A))\le 4M\|f\|\|g^*\|
\end{equation}
(see, e.g., \cite{Markin2004(1),Markin2004(2)}).

Also \cite{Markin2004(1),Markin2004(2)}, for any Borel measurable function $F:\C\to \C$, arbitrary $f\in D(F(A))$ and $g^*\in X^*$, and each Borel set $\delta\subseteq \C$,
\begin{equation}\label{cond(ii)}
\int\limits_\delta|F(\lambda)|\,dv(f,g^*,\lambda)
\le 4M\|E_A(\delta)F(A)f\|\|g^*\|.
\end{equation}
In particular, for $\delta=\sigma(A)$,
\begin{equation}\label{cond(i)}
\int\limits_{\sigma(A)}|F(\lambda)|\,d v(f,g^*,\lambda)\le 4M\|F(A)f\|\|g^*\|.
\end{equation}

Observe that the constant $M\ge 1$ in \eqref{tv}--\eqref{cond(i)} is from 
\eqref{bounded}.

\section[Main Results]{Main Results}

\begin{thm}\label{Thm1} 
An arbitrary scalar type spectral operator $A$ in a complex Banach space $(X,\|\cdot\|)$ with spectral measure $E_A(\cdot)$ is non-hypercyclic and so is the collection $\left\{e^{tA}\right\}_{t\ge 0}$ of its exponentials, which, provided the spectrum of $A$ is located in a left half-plane
\[
\left\{\lambda\in \C\,\middle|\,\Rep\lambda\le \omega \right\}
\]
with some $\omega\in\R$, coincides with the $C_0$-semigroup generated by $A$. 
\end{thm}

\begin{proof}
Let $f\in C^\infty(A)\setminus \{0\}$ be arbitrary.

There are two possibilities: either
\[
E_A\left(\left\{\lambda\in \sigma(A)\,\middle|\, |\lambda|>1\right\}\right)f\neq 0
\]
or
\[
E_A\left(\left\{\lambda\in \sigma(A)\,\middle|\, |\lambda|>1\right\}\right)f=0.
\]

In the first case, as follows from the {\it Hahn-Banach Theorem} (see, e.g., \cite{Dun-SchI}), there exists a functional $g^*\in X^*\setminus \{0\}$ such that
\begin{equation*}
\langle E_A\left(\left\{\lambda\in \sigma(A)\,\middle|\, |\lambda|>1\right\}\right)f,g^*\rangle\neq 0
\end{equation*}
and hence, for any $n\in \Z_+$,
\begin{multline*}
\|A^nf\|
\\
\hfill
\text{by \eqref{cond(i)};}
\\
\shoveleft{
\ge \left[4M\|g^*\|\right]^{-1}\int\limits_{\sigma(A)}|\lambda|^n\,dv(f,g^*,\lambda)
\ge \left[4M\|g^*\|\right]^{-1}
\int\limits_{\{\lambda\in\sigma(A)||\lambda|>1\}}|\lambda|^n\,dv(f,g^*,\lambda)
}\\
\shoveleft{
\ge 
\left[4M\|g^*\|\right]^{-1}v(f,g^*,\{\lambda\in\sigma(A)||\lambda|>1\})
}\\
\ \ \
\ge 
\left[4M\|g^*\|\right]^{-1}
|\langle E_A(\{\lambda\in\sigma(A)||\lambda|>1\})f,g^*\rangle|
>0,
\hfill
\end{multline*}
which implies that the orbit $\left\{A^nf\right\}_{n\in\Z_+}$ of $f$ under $A$ cannot approximate the zero vector, and hence, is not dense in $(X,\|\cdot\|)$.

In the second case, since
\[
f=E_A\left(\left\{\lambda\in \sigma(A)\,\middle|\, |\lambda|>1\right\}\right)f
+
E_A\left(\left\{\lambda\in \sigma(A)\,\middle|\, |\lambda|\le 1\right\}\right)f,
\]
we infer that
\[
f=E_A\left(\left\{\lambda\in \sigma(A)\,\middle|\, |\lambda|\le 1\right\}\right)f\neq 0
\]
and hence, for any $n\in \Z_+$,
\begin{multline*}
\left\|A^nf\right\|
\\
\hfill
\text{by the properties of the \textit{operational calculus};}
\\
\shoveleft{
=
\left\|\int\limits_{\{\lambda\in\sigma(A)||\lambda|\le 1\}}
\lambda^n\,dE_A(\lambda)f\right\|
}\\
\hfill
\text{as follows from the \textit{Hahn-Banach Theorem};}
\\
\shoveleft{
=\sup_{\{g^*\in X^*|\|g^*\|=1\}}
\left|\left\langle
\int\limits_{\{\lambda\in\sigma(A)||\lambda|\le 1\}}
\lambda^n\,d E_A(\lambda)f,g^*\right\rangle
\right|
}\\
\hfill
\text{by the properties of the \textit{operational calculus};}
\\
\shoveleft{
= \sup_{\{g^*\in X^*|\|g^*\|=1\}}
\left|\int\limits_{\{\lambda\in\sigma(A)||\lambda|\le 1\}}
\lambda^n\,d\langle E_A(\lambda)f,g^*\rangle\right|
}\\
\shoveleft{
\le \sup_{\{g^*\in X^*|\|g^*\|=1\}}\int\limits_{\{\lambda\in\sigma(A)||\lambda|\le 1\}}
|\lambda|^n\,dv(f,g^*,\lambda) 
}\\
\shoveleft{
\le \sup_{\{g^*\in X^*|\|g^*\|=1\}}\int\limits_{\{\lambda\in\sigma(A)||\lambda|\le 1\}}
1\,dv(f,g^*,\lambda) 
}\\
\hfill
\text{by \eqref{cond(ii)} with $F(\lambda)\equiv 1$};
\\
\shoveleft{
\le \sup_{\{g^*\in X^*|\|g^*\|=1\}}4M\left\|E_A(\{\lambda\in\sigma(A)|
|\lambda|\le 1\})f\right\|\|g^*\|
}\\
\ \ \
= 4M\left\|E_A(\{\lambda\in\sigma(A)|
|\lambda|\le 1\})f\right\|.
\hfill
\end{multline*}
which also implies that the orbit
$\left\{A^nf\right\}_{n\in\Z_+}$ of $f$ under $A$, being bounded, is not dense in $(X,\|\cdot\|)$ and completes the proof for the case of the operator.

\smallskip
Now, let us consider the case of the exponential collection $\left\{e^{tA}\right\}_{t\ge 0}$ assuming that $\displaystyle f \in \bigcap_{t\ge 0}D(e^{tA})\setminus \{0\}$ is arbitrary.

There are two possibilities: either
\[
E_A\left(\left\{\lambda\in \sigma(A)\,\middle|\, \Rep\lambda>0\right\}\right)f\neq 0
\]
or
\[
E_A\left(\left\{\lambda\in \sigma(A)\,\middle|\, \Rep\lambda>0\right\}\right)f=0.
\]

In the first case, as follows from the {\it Hahn-Banach Theorem}, there exists a functional $g^*\in X^*\setminus \{0\}$ such that
\begin{equation*}
\langle E_A\left(\left\{\lambda\in \sigma(A)\,\middle|\, \Rep\lambda>0\right\}\right)f,g^*\rangle\neq 0
\end{equation*}
and hence, for any $t\ge 0$,
\begin{multline*}
\|e^{tA}f\|
\\
\hfill
\text{by \eqref{cond(i)};}
\\
\shoveleft{
\ge \left[4M\|g^*\|\right]^{-1}\int\limits_{\sigma(A)}\left|e^{t\lambda}\right|\,dv(f,g^*,\lambda)
}\\
\shoveleft{
\ge \left[4M\|g^*\|\right]^{-1}
\int\limits_{\{\lambda\in\sigma(A)|\Rep\lambda>0\}}e^{t\Rep\lambda}\,dv(f,g^*,\lambda)
}\\
\hfill
\text{since for $t\ge 0$ and $\lambda\in \sigma(A)$ with $\Rep\lambda>0$, $e^{t\Rep\lambda}\ge 1$;}
\\
\shoveleft{
\ge 
\left[4M\|g^*\|\right]^{-1}v(f,g^*,\{\lambda\in\sigma(A)|\Rep\lambda>0\})
}\\
\ \ \,
\ge 
\left[4M\|g^*\|\right]^{-1}
|\langle E_A(\{\lambda\in\sigma(A)|\Rep\lambda>0\})f,g^*\rangle|
>0,
\hfill
\end{multline*}
which implies that the orbit $\left\{e^{tA}f\right\}_{t\ge 0}$ of $f$ cannot approximate the zero vector, and hence, is not dense in $(X,\|\cdot\|)$.

In the second case, since
\[
f=E_A\left(\left\{\lambda\in \sigma(A)\,\middle|\, \Rep\lambda>0\right\}\right)f
+
E_A\left(\left\{\lambda\in \sigma(A)\,\middle|\, \Rep\lambda\le 0\right\}\right)f,
\]
we infer that
\[
f=E_A\left(\left\{\lambda\in \sigma(A)\,\middle|\, \Rep\lambda\le 0\right\}\right)f\neq 0
\]
and hence, for any $t\ge 0$,
\begin{multline*}
\left\|e^{tA}f\right\|
\\
\hfill
\text{by the properties of the \textit{operational calculus};}
\\
\shoveleft{
=
\left\|\int\limits_{\{\lambda\in\sigma(A)|\Rep\lambda\le 0\}}
e^{t\lambda}\,dE_A(\lambda)f\right\|
}\\
\hfill
\text{as follows from the \textit{Hahn-Banach Theorem};}
\\
\shoveleft{
=\sup_{\{g^*\in X^*|\|g^*\|=1\}}
\left|\left\langle
\int\limits_{\{\lambda\in\sigma(A)|\Rep\lambda\le 0\}}
e^{t\lambda}\,d E_A(\lambda)f,g^*\right\rangle
\right|
}\\
\hfill
\text{by the properties of the \textit{operational calculus};}
\\
\shoveleft{
= \sup_{\{g^*\in X^*|\|g^*\|=1\}}
\left|\int\limits_{\{\lambda\in\sigma(A)|\Rep\lambda\le 0\}}
e^{t\lambda}\,d\langle E_A(\lambda)f,g^*\rangle\right|
}\\
\shoveleft{
\le \sup_{\{g^*\in X^*|\|g^*\|=1\}}\int\limits_{\{\lambda\in\sigma(A)|\Rep\lambda\le 0\}}
\left|e^{t\lambda}\right|\,dv(f,g^*,\lambda) 
}\\
\shoveleft{
=\sup_{\{g^*\in X^*|\|g^*\|=1\}}\int\limits_{\{\lambda\in\sigma(A)|\Rep\lambda\le 0\}}
e^{t\Rep\lambda}\,dv(f,g^*,\lambda) 
}\\
\hfill
\text{since for $t\ge 0$ and $\lambda\in \sigma(A)$ with $\Rep\lambda\le 0$, $e^{t\Rep\lambda}\le 1$;}
\\
\shoveleft{
\le \sup_{\{g^*\in X^*|\|g^*\|=1\}}\int\limits_{\{\lambda\in\sigma(A)|\Rep\lambda\le 0\}}
1\,dv(f,g^*,\lambda) 
}\\
\hfill
\text{by \eqref{cond(ii)} with $F(\lambda)\equiv 1$};
\\
\shoveleft{
\le \sup_{\{g^*\in X^*|\|g^*\|=1\}}4M\left\|E_A(\{\lambda\in\sigma(A)|
\Rep\lambda\le 0\})f\right\|\|g^*\|
}\\
\ \ \
= 4M\left\|E_A(\{\lambda\in\sigma(A)|
\Rep\lambda\le 0\})f\right\|,
\hfill
\end{multline*}
which also implies that the orbit
$\left\{e^{tA}f\right\}_{t\ge 0}$ of $f$, being bounded, is not dense in $(X,\|\cdot\|)$ and completes the entire proof.
\end{proof}

\begin{rem}
Now, \cite[Theorem $1$]{Mark-Sich2019(1)} is the important particular case of Theorem \ref{Thm1} for
a (bounded or unbounded) \textit{normal operator} in a complex Hilbert space.
\end{rem} 

If further for a scalar type spectral operator $A$ in a complex Banach space $(X,\|\cdot\|)$, we have the inclusion:
\[
\sigma(A)\subseteq i\R,
\]
by {\cite[Theorem XVIII.$2.11$ (c)]{Dun-SchIII}}, for any $t\in \R$,
\[
\|e^{tA}\|
=\left\|\int\limits_{\sigma(A)}
e^{t\lambda}\,dE_A(\lambda)\right\|\le 4M\sup_{\lambda\in\sigma(A)}\left|e^{t\lambda}\right|
=4M\sup_{\lambda\in\sigma(A)}e^{t\Rep\lambda}=4M,
\]
where the constant $M\ge 1$ is from 
\eqref{bounded}. Therefore, the strongly continuous  group $\left\{e^{tA}\right\}_{t\in \R}$ of bounded linear operators generated by $A$ is \textit{bounded} (cf. \cite{Berkson1966}), which implies that every orbit $\left\{e^{tA}f\right\}_{t\in \R}$, $f\in X$,
is bounded, and hence, is not dense in $(X,\|\cdot\|)$. Thus, we arrive at the following

\begin{prop}\label{Prop1} 
For a scalar type spectral operator $A$ in a complex Banach space $(X,\|\cdot\|)$ with $\sigma(A)\subseteq i\R$, the strongly continuous  group $\left\{e^{tA}\right\}_{t\in \R}$ of bounded linear operators generated by $A$ is bounded, and hence, non-hypercyclic. 
\end{prop}

As is known \cite{Stone1932}, for an \textit{anti-self-adjoint operator} $A$ in a complex Hilbert space, $\sigma(A)\subseteq i\R$ and the generated by $A$ strongly continuous operator group
$\left\{e^{tA}\right\}_{t\in \R}$ is \textit{unitary}, which, in particular, implies that
\[
\|e^{tA}\|=1,\ t\in \R.
\]

Thus, from Theorem \ref{Thm1} (see also \cite[Theorem $1$]{Mark-Sich2019(1)}) and Proposition \ref{Prop1}, we derive the following corollary.

\begin{cor}[The Case of an Anti-Self-Adjoint Operator]\label{Cor1}\ \\
An anti-self-adjoint operator $A$ in a complex Hilbert space is non-hypercyclic and so is the generated by $A$ strongly continuous unitary operator group $\left\{e^{tA}\right\}_{t\in \R}$. 
\end{cor}

\section{An Application}

Since, in the complex Hilbert space $L_2({\mathbb R})$, the differentiation operator $A:=\dfrac{d}{dx}$ with the domain
\[
W_2^1({\mathbb R}):=\left\{f\in L_2(\R)\middle|f(\cdot)\ 
\text{is \textit{absolutely continuous} on $\R$ with}\ f'\in L_2(\R) \right\}
\] 
is \textit{anti-self-adjoint} (see, e.g., \cite{Akh-Glaz}), by Corollary \ref{Cor1}, we obtain

\begin{cor}[The Case of Differentiation Operator]\label{Cor2}\ \\
In the complex Hilbert space $L_2({\mathbb R})$, the differentiation operator $A:=\dfrac{d}{dx}$ with the domain $D(A):=W_2^1({\mathbb R})$ is non-hypercyclic and so is the left-translation strongly continuous unitary operator group generated by $A$.
\end{cor}

\begin{rem}
In a different setting, the situation with the differentiation operator can be vastly different
(cf. \cite[Example $2.21$]{Grosse-Erdmann-Manguillot}, \cite[Corollary $2.3$]{B-Ch-S2001}, \cite[Corollary $4.1$]{E-H2005}, and \cite[Theorem $3.1$]{B-B-T2008}).
\end{rem}

\section{Concluding Remark}

The exponentials given by \eqref{expf1} describing all \textit{weak/mild solutions} of evolution equation \eqref{+} (see Preliminaries), Theorem \ref{Thm1}, in particular, implies that such an equation is void of chaos (see \cite{Grosse-Erdmann-Manguillot}). By Proposition \ref{Prop1} (see also Preliminaries), the same is true for evolution equation \eqref{1} provided $\sigma(A)\subseteq i\R$.

\section{Acknowledgments}

I would like to express sincere gratitude to Dr.~Oscar Vega of the Department of Mathematics, California State University, Fresno, for his gift of book \cite{Guirao-Montesinos-Zizler} reading which inspired the above findings. My utmost appreciation goes to Professor Karl Grosse-Erdmann of the Universit\'e de Mons Institut de Math\'ematique  for his interest in my endeavors in the realm of unbounded hypercyclicity and chaoticity, turning my attention to the existing research on the subject of unbounded linear hypercyclicity and chaos, and kindly communicating certain relevant references.

 
\end{document}